\documentclass[12pt]{article}

\nonstopmode

\usepackage[left=1cm,right=1cm,top=2cm,bottom=2cm]{geometry}
\usepackage[utf8]{inputenc}
\usepackage[english]{babel}
\usepackage[T1]{fontenc}
\usepackage{amsmath}
\usepackage{amsthm}
\usepackage{amsfonts}
\usepackage{amssymb}
\usepackage{color}
\usepackage{multicol}
\usepackage{times}
\usepackage{setspace}
\usepackage{hyperref}
\usepackage{hyperref}
\hypersetup{
colorlinks=true, 
breaklinks=true, 
urlcolor= blue, 
linkcolor= blue, 
citecolor=blue, 
}
\usepackage{graphicx}
\usepackage{lscape}
\usepackage{setspace}

\newcommand{\norm}[1]{\left\lVert #1 \right\rVert}

\newtheorem{theorem}{Theorem}[section]
\newtheorem{lemma}{Lemma}[section]

\newtheorem{cor}{Corollary}[section]

\newcommand{\bZ}{\ensuremath{\mathbb{Z}}}
\newcommand{\bQ}{\ensuremath{\mathbb{Q}}}
\newcommand{\bK}{\ensuremath{\mathbb{K}}}

\newcommand{\bL}{\ensuremath{\mathbb{L}}}

\newcommand{\klammern}[4][]%
{\ifthenelse{\equal{#1}{}}{\left#2}{\csname#1\endcsname#2}%
#4\ifthenelse{\equal{#1}{}}{\right#3}{\csname#1\endcsname#3}}

\title{Narayana's cows numbers which are concatenations of three repdigits in base $\rho$} 

\author{Pagdame Tiebekabe$^{1,\, 2,\, 3}$, Kou\`essi Norbert Ad\'edji$^{4}$, \\ Nadjime Pindra$^{5}$, and Mahouton Norbert Hounkonnou$^{*, \,2,\, 6}$}
\date{}
\begin{document}

\maketitle
\begin{abstract}
\noindent Narayana's sequence is a ternary recurrent sequence defined by the recurrence relation
		$\mathcal{N}_n=\mathcal{N}_{n-1}+\mathcal{N}_{n-3}$ with initial terms $\mathcal{N}_0=0$ and $\mathcal{N}_1=\mathcal{N}_2=\mathcal{N}_3=1$. Let $\rho\geqslant2$ be a positive integer. In this study, it is proved that the $n$th Narayana's number $ \mathcal{N}_n$ which is concatenations of three repdigits in base $\rho$ satisfies $n<5.6\cdot 10^{48}\cdot \log^7\rho$. Moreover, it is shown that the largest Narayana's number which is concatenations  of three repdigits in base $\rho$ with $1 \leqslant \rho \leqslant 10$ is $58425=\mathcal{N}_{31}=\overline{3332200}_5=\overline{332223}_7 .$
\end{abstract}

\section{Introduction}
	The power of linear forms in logarithms for solving Diophantine exponential equations in general, and Diophantine equations in linear recurrent sequences in particular, is well established. Since the introduction of this powerful tool by the British mathematician Alan Baker in $1966$ by  proving a landmark result about linear forms in the logarithms of algebraic numbers, which helped to earn him the Fields Medal in $1970$, several exponential Diophantine equations have been solved. Since few years now, researchers in number theory have been interested in the resolution of Diophantine equations in linear recurrent sequences. Variants of these equations have been intensely studied. The determination of the terms of linear recurrent sequences which are product or concatenations of two or three repdigits have been widely studied. For details of recent works related to this problem, we refer the reader to \cite{5,6,7,8,9,10,11,12,13}. These sequences might be useful for solving some problem in discrete classical mechanics \cite{hkn}.

	\noindent Indeed, linear recurrent sequences and the distribution of repdigits or concatenated sequences in physical phenomena is an interesting area to explore. While specific examples may not be prevalent in the current scientific literature, the idea is aligned with the interdisciplinary nature of mathematics and physics. 
	
	Let $u_n$ represent the displacement of the point on the string at time $n$. The displacement at any given time depends on the displacements at previous times. Mathematically, this relationship can be expressed as a linear recurrence relation:
	$u_{n+2}=2u_{n+1}-u_n$. This linear recurrence relation is a simplified model that describes the motion of a point on the string based on its displacements at the two preceding time steps.
	
	Repdigits exhibit periodicity by nature. In physical systems that exhibit periodic behavior, such as oscillations or wave phenomena \cite{hg}, the occurrence of repdigits or concatenated sequences might be correlated with specific phases or cycles.
	
	Besides, the arrangement of atoms in crystals follows specific patterns.
	 The use of concatenated sequences could potentially provide a unique way to describe or analyze certain crystal structures. The repetition in repdigits might be indicative of regularities in atomic arrangements.
	
	\noindent In quantum mechanics, the properties of repdigits might find application in describing certain quantum states. The concatenated sequences could be used in the context of quantum information theory 
	or in the representation of quantum states with specific symmetries.
	
	\noindent The study of chaos theory involves the exploration of deterministic systems exhibiting sensitive dependence on initial conditions. The presence of repdigits or concatenated sequences might reveal certain order or chaotic patterns within these systems.
	
	\noindent While not directly related to repdigits, fractal geometry often involves self-similar structures that exhibit repetitive patterns at different scales. Repdigits could potentially be used as building blocks or motifs in constructing fractal geometries relevant to certain physical systems, such as fractal patterns in fluid dynamics or turbulent flows. Classical geometric constructions with ruler and compass often involve questions about constructibility using only a finite number of steps. Many of these constructions relate to algebraic numbers and their properties, connecting geometry with algebra and number theory.
	
	\noindent Further, in astrophysics, certain periodic phenomena in the cosmos, such as pulsars or orbital mechanics, might have correlations with repdigits or concatenated sequences. The study of cosmic rhythms or repeating patterns could involve mathematical representations using such sequences.
	
	\noindent Our work extends previous investigations on this topic \cite{5,6,7,8,9,10,11,12,13}. Following what was done by Mahadi Ddamulira {\it et al} \cite{5}, we study the problem of writing all Narayana numbers which are concatenations of three repdigits in base $\rho$ with $\rho \geq 2$. To be precise, we find all solutions of the Diophantine equation
	\begin{equation}\label{eqfondamentale1}
		\mathcal{N}_n=\overline{\underbrace{d_1\ldots d_1}_{\ell~times}}\overline{\underbrace{d_2\ldots d_2}_{m~times}}\overline{\underbrace{d_3\ldots d_3}_{k~times}}
	\end{equation}
	in non-negative integers $n, d_1, d_2, d_3, \ell, m, k$ with $n \geqslant 0, 1 \leqslant k \leqslant m \leqslant \ell$ and $d_1, d_2, d_3 \in \{0, 1, 2, 3, \ldots, 9\}, d_1 > 0$. Contrary to most previous works on this question done in base $10$, we determine an upper bound of $n$ and $\ell+m+k$ in terms of $\rho$ with $\rho \geqslant 2$. For application, we find in section $4$ all Narayana's cows numbers that are concatenations of three repdigits in  the cases $2 \leqslant \rho  \leqslant 10$.
	The main results are as follows :
	
	\begin{theorem}\label{main}
		Let $\rho\geqslant 2$ be an integer. All solutions to the Diophantine equation \eqref{eqfondamentale1} satisfy
		$$
		\ell+m+k< n<5.6\cdot 10^{48}\cdot \log^7\rho.
		$$
	\end{theorem}
	Moreover, the above result implies the following :
	\begin{cor}
		The Diophantine equation \eqref{eqfondamentale1}  has only finitely many solutions in positive  integers  $n, d_1, d_2, d_3, \ell, m$ and $k.$
	\end{cor}
	By considering the cases $2 \leqslant \rho \leqslant 10,$ we get the following result.
	\begin{theorem}\label{Application}
		The only Nayarana numbers which are concatenations of three repdigits in base $\rho$ with $2\le \rho\le 10$ are
		\[
		4, 6, 9, 13, 19, 28, 41, 60, 88, 129, 189, 277, 406, 595, 872, 1278,  1873, 2745, 4023, 18560, 58425.
		\]
		More precisely, we have: $ 4=\mathcal{N}_6=\overline{100}_2,$ $6=\mathcal{N}_7=\overline{110}_2,$ $9=\mathcal{N}_8=\overline{101}_2=\overline{100}_3,$ $13=\mathcal{N}_9=\overline{111}_3=\overline{1101}_2,$ $19=\mathcal{N}_{10}=\overline{10011}_2=\overline{201}_3=\overline{103}_4, $
		$$
		\begin{array}{ccccl}
			28&=&\mathcal{N}_{11}&=&\overline{1001}_3=\overline{130}_4=\overline{103}_5=\overline{11100}_2,\\
			41&=&\mathcal{N}_{12}&=&\overline{1112}_3=\overline{221}_4=\overline{131}_5,\\
			60&=&\mathcal{N}_{13}&=&\overline{111100}_2=\overline{330}_4=\overline{220}_5=\overline{140}_6=\overline{114}_7,\\
			88&=&\mathcal{N}_{14}&=&\overline{323}_5=\overline{224}_6=\overline{154}_7=\overline{130}_8=\overline{1120}_4=\overline{107}_9,\\
			129&=&\mathcal{N}_{15}&=&\overline{333}_6=\overline{243}_7=\overline{201}_8=\overline{1004}_5=\overline{2001}_4=\overline{153}_9=\overline{10000001}_2=\overline{129}_{10},\\
			189&=&\mathcal{N}_{16}&=&\overline{2331}_4=\overline{21000}_3=\overline{360}_7=\overline{230}_9=\overline{1224}_5=\overline{513}_6=\overline{275}_8=\overline{189}_{10},\\
			277&=&\mathcal{N}_{17}&=&\overline{1141}_6=\overline{544}_7=\overline{425}_8=\overline{337}_9=\overline{10111}_4,\\
			406&=&\mathcal{N}_{18}&=&\overline{3111}_5=\overline{626}_8=\overline{1120}_7=\overline{501}_9=\overline{406}_{10},\\
			595&=&\mathcal{N}_{19}&=&\overline{1123}_8=\overline{731}_9=\overline{595}_{10},\\
			872&=&\mathcal{N}_{20}&=&\overline{11442}_5=\overline{1168}_9=\overline{1550}_8=\overline{872}_{10},\\
			1278&=&\mathcal{N}_{21}&=&\overline{5530}_6,\\
			1873&=&\mathcal{N}_{22}&=&\overline{24443}_5=\overline{2511}_9,\\
			2745&=&\mathcal{N}_{23}&=&\overline{11001}_7,\\
			4023&=&\mathcal{N}_{24}&=&\overline{7667}_8,\\
			18560&=&\mathcal{N}_{28}&=&\overline{44200}_8,\\
			58425&=&\mathcal{N}_{31}&=&\overline{3332200}_5=\overline{332223}_7.
		\end{array}
		$$
	\end{theorem}
	
	For the proof of Theorem \ref{main}, we use $n> 560$ by writing Diophantine \eqref{eqfondamentale1} in three different ways to get three linear forms in logarithms. Next, we apply lower bound for linear forms in logarithms to obtain an upper bound on $n$. To do this, we need some preliminary results, which are discussed in the following section.
	

	\section{Preliminary results}
	
	\subsection{Some properties of Narayana's cows sequence}
	Narayana's cows sequence comes from a problem with cows proposed by the Indian mathematician Narayana in the $14-$th century. In this problem, we assume that there is a cow at the beginning and each cow produces a calf every year from the $4$ years old. Narayana's cow problem counts the number of calves produced each year \cite{1}. This problem seems to be similar to the Fibonacci rabbit problem. So also the
	answers are known as Narayana and Fibonacci sequences.
	
	Narayana’s cows sequence (A000930 in OEIS) satisfies a third-order recurrence relation :
	$$
	\mathcal{N}_n=\mathcal{N}_{n-1}+\mathcal{N}_{n-3}\, \text{for}\, n\geqslant 3.
	$$
	This has the initial values $\mathcal{N}_0$ and $\mathcal{N}_1=\mathcal{N}_2=\mathcal{N}_3=1$ \cite{1}. Explicitly, the characteristic polynomial of $\mathcal{N}_n$ is :
	$$
	\varphi(x)=x^3-x-1,
	$$
	and the characteristic roots are:
	\begin{align}
		\alpha=&\dfrac{1}{3}\left( \sqrt[3]{\dfrac{1}{2}(29-3\sqrt{93})} + \sqrt[3]{\dfrac{1}{2}(3\sqrt{93}+29)}+1  \right),\\
		\beta = & \dfrac{1}{3}-\dfrac{1}{6}\left( 1-i\sqrt{3} \right)\sqrt[3]{\dfrac{1}{2}(29-3\sqrt{93})}-\dfrac{1}{6}\left( 1+i\sqrt{3} \right)\sqrt[3]{\dfrac{1}{2}(3\sqrt{93}+29)},\\
		\gamma = & \dfrac{1}{3}-\dfrac{1}{6}\left( 1+i\sqrt{3} \right)\sqrt[3]{\dfrac{1}{2}(29-3\sqrt{93})}-\dfrac{1}{6}\left( 1-i\sqrt{3} \right)\sqrt[3]{\dfrac{1}{2}(3\sqrt{93}+29)}.
	\end{align}
	Then, the Narayana sequence can be obtained by Binet’s formula:
	\begin{equation}\label{eq4}
		\mathcal{N}_n=a\alpha^n+b\beta^n+c\gamma^n.
	\end{equation}
	For $n\in \mathbb{Z}_{\geqslant 0}$ the generating function of the Narayana's cows sequence is 
	$$
	g(x)=\dfrac{1}{1-x-x^3}=\sum_{n=0}^{\infty} \mathcal{N}_{n+1}x^n.
	$$
	With the Vieta theorem, we have:
	\begin{equation}\label{eq5}
		\left\{
		\begin{array}{lllll}
			\alpha+ \beta+\gamma=1,\\
			\alpha \beta+\beta \gamma+ \alpha \gamma=0,\\
			\alpha \beta \gamma=1.
		\end{array}
		\right.
	\end{equation}
	From formula (\ref{eq4}), we obtain
	\begin{align*}
		\mathcal{N}_0 &= a+b+c=0,\\
		\mathcal{N}_1 &=a\alpha+b\beta + c\gamma=1,\\
		\mathcal{N}_2 &=a\alpha^2+b\beta^2+c\gamma^2=1,
	\end{align*}
	which implies
	$$
	a=\dfrac{1-\beta-\gamma}{(\alpha-\beta)(\alpha-\gamma)},\quad b=\dfrac{1-\alpha-\gamma}{(\beta-\alpha)(\beta-\gamma)}, \quad\text{and}\quad c=\dfrac{1-\alpha-\beta}{(\gamma-\beta)(\gamma-\alpha)}.
	$$
	Also, from formula \eqref{eq5}, we can simplify $a, b$ and $c$ and obtain 
	\begin{equation*}
		a=\dfrac{\alpha}{\alpha^2-\alpha \beta-\alpha \gamma+\beta \gamma}=\dfrac{\alpha}{\alpha^2+2\beta \gamma}=\dfrac{\alpha^2}{\alpha^3+2}
	\end{equation*}
	\begin{equation}\label{eq6}
		b=\dfrac{\beta^2}{\beta^3+2}\quad \text{and}\quad c=\dfrac{\gamma^2}{\gamma^3+2}
	\end{equation}
	and the minimal polynomial of $a$ over integers is $31x^3-3x-1$. Setting $\Pi(n)=\mathcal{N}_n-a\alpha^n=b\beta^n+c\gamma^n$, we notice that $|\Pi(n)|<\dfrac{1}{\alpha^{n/2}}$ for all $n>1$.
	
	The Narayana's sequence was originally defined at positive indices. Actually, it can be extended to negative indices by defining 
	\begin{equation}\label{eq7}
		\mathcal{N}_{-n}=\dfrac{a}{\alpha^n}+\dfrac{b}{\beta^n}+\dfrac{c}{\gamma^n}.
	\end{equation}
	So, the following recurrence relation holds for all integral indices
	\begin{equation}\label{eq8}
		\mathcal{N}_n=\mathcal{N}_{n-1}+\mathcal{N}_{n-3}\quad \text{for}\quad n\in \mathbb{Z}.
	\end{equation}
	Through a simple computation, the first few terms of $\mathcal{N}_n$ at negative indices can be obtained from formulas (\ref{eq6}) and (\ref{eq7}), so that $\mathcal{N}_{-1}=0,\, \mathcal{N}_{-2}=1,\, \mathcal{N}_{-3}=0,\, \mathcal{N}_{-4}=-1$, which also satisfies relation (\ref{eq8}). 
	
	The characteristic polynomial for $(\mathcal{N}_n)_{n\geqslant 0}$ is irreducible in $\mathbb{Q}[x]$. We note that it has a real zero $\alpha (> 1)$ and two conjugate complex zeros $\beta$ and $\gamma$ with $|\beta|=|\gamma|<1$. In fact, using Computer Algebra, we have $\alpha \approx 1.46557$. We also have the following properties of $(\mathcal{N}_n)_{n\geqslant 0}$.
	
	\begin{lemma}\label{lem1}
		For the sequence $(\mathcal{N}_n)_{n\geqslant 0}$, we have 
		$$
		\alpha^{n-2} \leqslant \mathcal{N}_n \leqslant \alpha^{n-1},\quad\text{for}\quad n\geqslant 1.
		$$ 
	\end{lemma}
	
	\begin{proof}
		Using induction, one can easily prove Lemma~\ref{lem1}. 
	\end{proof}
	Let $\mathbb{K}:=\mathbb{Q}(\alpha,\beta)$ be the splitting field of the characteristic polynomial $\varphi$ over $\mathbb{Q}$. Then $[\mathbb{K}:\mathbb{Q}]=6$ and $[\mathbb{Q}(\alpha):\mathbb{Q}]=3.$ The Galois group of $\mathbb{K}/\mathbb{Q}$ is given by 
	$$
	\mathcal{G}:= \text{Gal}(\mathbb{K}/\mathbb{Q})\cong \{(1), (\alpha\beta), (\alpha \gamma), (\beta \gamma), (\alpha\beta\gamma)\}\cong S_3.
	$$
	We identify the automorphisms of $\mathcal{G}$ with the permutation group of the zeroes of $\varphi$. We highlight the permutation $(\alpha\beta)$, corresponding to the automorphism $\sigma: \alpha\mapsto \beta, \beta\mapsto \alpha, \gamma\mapsto\gamma$, which we use later to obtain a contradiction on the size of the absolute value of certain bound. 
	\subsection{Linear forms in Logarithms}
	
	We begin this section with a few reminders about logarithmic height of an algebraic number. Let $\eta$ be an algebraic number of degree $d,$ let $a_0 >0$ be the leading coefficient of its minimal polynomial over $\bZ$ and let $\eta=\eta^{(1)},\ldots,\eta^{(d)}$ denote its conjugates. The quantity defined by  
	\[
	h(\eta)= \frac{1}{d}\left(\log |a_0|+\sum_{j=1}^{d}\log\max\left(1,\left|\eta^{(j)} \right| \right) \right)
	\]
	is called the logarithmic height of $\eta.$ Some properties of height are as follows. For $\eta_1, \eta_2$ algebraic numbers and $m\in \bZ,$ we have
	\begin{align*}
		h(\eta_1 \pm \eta_2) &\leq h(\eta_1)+ h(\eta_2) +\log2,\\
		h(\eta_1\eta_2^{\pm 1}) &\leq h(\eta_1) + h(\eta_2),\\
		h(\eta_1^m)&=|m|h(\eta_1).
	\end{align*}
	If $\eta=\dfrac{p}{q}\in\bQ$ is a rational number in reduced form with $q>0,$ then the above definition reduces to $h(\eta)=\log(\max\{|p|,q\}).$
	We can now present the famous Matveev result used in this study. Thus, let $\bL$ be a real number field of degree $d_{\bL}$, $\eta_1,\ldots,\eta_s \in \bL$ and $b_1,\ldots,b_s \in \bZ  \setminus\{0\}.$ Let $B\ge \max\{|b_1|,\ldots,|b_s|\}$ and
	\[
	\Lambda=\eta_1^{b_1}\cdots\eta_s^{b_s}-1.
	\]
	Let $A_1,\ldots,A_s$ be real numbers with 
	\[
	A_i\ge \max\{d_{\bL} h(\eta_i), |\log\eta_i|, 0.16\},\quad i=1,2,\ldots,s.
	\]
	With the above notations, Matveev proved the following result \cite{3}.
	\begin{theorem}\label{Matveev}$\\$
		Assume that $\Lambda\neq 0.$ Then
		\[
		\log|\Lambda|>-1.4\cdot30^{s+3}\cdot s^{4.5}\cdot d_{\bL}^2 \cdot(1+\log d_{\bL})\cdot(1 +\log B)\cdot A_1\cdots A_s.
		\]
	\end{theorem}
	We also need the following result of \cite{SGL}.
	\begin{lemma}\label{lem3}
		Let $r\geqslant 1$ and $H>0$ be such that $H>(4r^2)^r$ and $H>L/(\log L)^r$. Then 
		$$
		L<2^rH(\log H)^r.
		$$
	\end{lemma}
	\subsection{Reduction method}
	
	The bounds on the variables obtained via Baker’s theory are too large for any computational
	purposes. To reduce the bounds, we use reduction method due to Dujella and Peth\H o (\cite{4}, Lemma 5a).
	For a real number $X$, we write $\norm{X} :=\min\{\mid X-n\mid :n\in \mathbb{Z}\}$ for the distance of $X$ to the nearest integer.
	
	\begin{lemma}[Dujella and Peth\H o, \cite{4}]\label{Dujella}
		Let $M$ be a positive integer, $p/q$ be a convergent of the continued fraction expansion of the irrational number $\tau$ such that $q>6M,$ and $A, B, \mu$ be some real numbers with $A>0$ and $B>1.$ Furthermore, let
		$$ 
		\varepsilon:=\parallel\mu q\parallel-M\cdot\parallel\tau q\parallel.
		$$
		If $\varepsilon>0,$ then there is no solution to the inequality
		\begin{equation}
			0<|u\tau-v+\mu|<AB^{-w}
		\end{equation}
		in positive integers $u,v$ and $w$ with
		$$
		u \leqslant M\; \mbox{and}\; w\geqslant \frac{\log(Aq/\varepsilon)}{\log B}.
		$$
	\end{lemma}
	The following results will also be used in our proof.
	\begin{lemma}\label{lem4}
		For any non-zero real number $x$, if $|e^x-1|<\dfrac{1}{2}$, then $|x|<2|e^x-1|$.
	\end{lemma}

	\section{Proof of Theorem \ref{main}}
	
	For the proof of Theorem \ref{main} we can assume that $n>560.$
	
	\subsection{The initial bound on $n$.}
	To begin with, we consider the Diophantine equation (\ref{eqfondamentale1}), and rewrite it as
	\begin{align*}
		\mathcal{N}_n=&\overline{\underbrace{d_1\ldots d_1}_{\ell~times}}\overline{\underbrace{d_2\ldots d_2}_{m~times}}\overline{\underbrace{d_3\ldots d_3}_{k~times}}\\
		=& \overline{\underbrace{d_1\ldots d_1}_{\ell~times}}\cdot \rho^{m+k} + \overline{\underbrace{d_2\ldots d_2}_{m~times}}\cdot \rho^k + \overline{\underbrace{d_3\ldots d_3}_{k~times}}\\
		=& d_1\left(\dfrac{\rho^\ell-1}{\rho-1} \right)\cdot \rho^{m+k}+ d_2 \left(\dfrac{\rho^m-1}{\rho-1} \right)\cdot \rho^k+ d_3\left(\dfrac{\rho^k-1}{\rho-1}\right).
	\end{align*}
	Therefore, we have
	\begin{equation}\label{eqfondamentale}
		\mathcal{N}_n=\dfrac{1}{\rho-1}\left( d_1 \rho^{\ell+m+k} - (d_1-d_2)\rho^{m+k}-(d_2-d_3)\rho^k-d_3\right).
	\end{equation}
	We state and prove the following lemma which gives a relation between $n$ and $l+m+k$ of (\ref{eqfondamentale}).
	
	\begin{lemma}\label{lem2}
		All solutions to the Diophantine equation (\ref{eqfondamentale1}) satisfy 
		$$
		(\ell+m+k-1)\log\rho +\log \alpha<n\log \alpha< (\ell+m+k)\log \rho+ 1.
		$$
	\end{lemma}
	\begin{proof}
		From (\ref{eqfondamentale1}) and Lemma \ref{lem1}, we get 
		$$
		\alpha^{n-2} \leqslant \mathcal{N}_n<\rho^{\ell+m+k}.
		$$
		Taking logarithm on both sides, we have
		$$(n-2)\log \alpha< (\ell+m+k)\log \rho$$
		which leads to 
		\begin{equation}\label{eq9}
			n\log \alpha< (\ell+m+k)\log\rho+ 2\log \alpha< (\ell+m+k)\log\rho+ 1.
		\end{equation}
		On the other hand, for the lower bound, \eqref{eqfondamentale} implies that 
		$$
		\rho^{\ell+m+k-1}< \mathcal{N}_n \leqslant \alpha^{n-1}.
		$$
		Taking logarithm on both sides, we get
		$$
		(\ell+m+k-1)\log\rho <(n-1)\log \alpha,
		$$
		which leads to 
		\begin{equation}\label{eq10}
			(\ell+m+k-1)\log\rho +\log \alpha< n\log \alpha.
		\end{equation}
		Comparing (\ref{eq9}) and (\ref{eq10}) completes the proof of Lemma \ref{lem2}.
	\end{proof}
	
	\subsubsection{Upper bound for $\ell \log\rho$ in terms of $n$}
	\begin{lemma}
		All solutions to (\ref{eqfondamentale1}) satisfy
		$$
		\ell \log\rho< 3.74\cdot 10^{13}\cdot \log^2\rho\cdot (1+\log n).
		$$
	\end{lemma}
	\begin{proof}
		Using (\ref{eq4}) and (\ref{eqfondamentale}), we have 
		$$
		a\alpha^n+b\beta^n+c\gamma^n=\dfrac{1}{\rho-1}\left( d_1\cdot \rho^{\ell +m+k}-(d_1-d_2)\cdot \rho^{m+k}-(d_2-d_3)\cdot \rho^k-d_3  \right).
		$$
		Equivalently 
		$$
		(\rho-1)a\alpha^n-d_1\cdot \rho^{\ell+m+k}=-(\rho-1)\Pi(n)-(d_1-d_2)\cdot \rho^{m+k}-(d_2-d_3)\cdot \rho^k-d_3.
		$$
		Thus, we have that 
		\begin{align*}
			\left|  (\rho-1)a\alpha^n-d_1\cdot \rho^{\ell+m+k}\right|&= \left| -(\rho-1)\Pi(n)-(d_1-d_2)\cdot \rho^{m+k}-(d_2-d_3)\cdot  \rho^k-d_3 \right|\\
			& \leqslant (\rho-1)\cdot \alpha^{-n/2}+(\rho-1)\cdot \rho^{m+k}+(\rho-1)\cdot \rho^k+(\rho-1)\\
			&< 3(\rho-1)\cdot \rho^{m+k},
		\end{align*}
		where we used the fact that $n>560$.
		
		\noindent Dividing both sides of the inequality by $d_1\cdot \rho^{\ell+m+k}$ gives 
		\begin{equation}\label{eq11}
			\left|  \left( \dfrac{(\rho-1)a}{d_1}\right)\cdot \alpha^n\cdot \rho^{-(\ell+m+k)}-1  \right|<\dfrac{3(\rho-1)\cdot \rho^{m+k}}{d_1\cdot \rho^{\ell+m+k}}<\dfrac{3}{\rho^{\ell-1}}.
		\end{equation}
		Let 
		\begin{equation}\label{eq12}
			\Lambda_1:= \left( \dfrac{(\rho-1)a}{d_1}\right)\cdot \alpha^n\cdot \rho^{-(\ell+m+k)}-1.
		\end{equation}
		We then proceed to apply Theorem \ref{Matveev} on (\ref{eq12}). We have first to observe that $\Lambda_1\neq 0$. Indeed, if it were zero, we would then get that
		$$
		a\alpha^n=\dfrac{d_1}{\rho-1}\cdot \rho^{\ell+m+k}.
		$$
		In this case therefore, applying the automorphism $\sigma$ of the Galois group $\mathcal{G}$ on both sides of the preceeding equation and taking absolute values, we obtain 
		$$
		\left| \left(  \dfrac{d_1}{\rho-1}\cdot \rho^{\ell+m+k}\right)\right|=|\sigma(a\alpha^n)|=|c\gamma^n|<1,
		$$
		which is false. Thus, we have $\Lambda_1\neq 0$. Theorem \ref{Matveev} is then applied on (\ref{eq12}) with the following parameters : 
		$$
		\eta_1:=\dfrac{(\rho-1)a}{d_1},\, \eta_2:=\alpha,\, \eta_3:=\rho, \, b_1:=1,\, b_2:=n,\,b_3:=-\ell-m-k,\,s:=3.
		$$
		From Lemma \ref{lem2}, we have $\ell+m+k<n$. Consequently, we choose $B:=n$. Notice that $\mathbb{K}:=\mathbb{Q}(\eta_1,\eta_2,\eta_3)=\mathbb{Q}(\alpha)$, since $a=\dfrac{\alpha^2}{\alpha^3+2}$. Moreover its minimal polynomial over integers is $31x^3-3x-1.$ Therefore, $d_{\mathbb{K}}:=[\bK:\bQ]=3$. 
		
		\noindent Using the properties of the logarithmic height, we estimate $h(\eta_1)$ as follows : 
		\begin{align*}
			h(\eta_1)&=h\left( \dfrac{(\rho-1)a}{d_1} \right) \leqslant h(\rho-1)+h(a)+h(d_1)\\
			& \leqslant 2\log \rho+\dfrac{1}{3}\log 31 \leqslant 4\log \rho,
		\end{align*}
		which holds for $\rho \geqslant2.$ Similarly, we have $h(\eta_2)=h(\alpha)=\dfrac{\log \alpha}{3}$ and $h(\eta_3)=h(\rho)=\log\rho$. Therefore, we choose 
		$$
		A_1:=12\log\rho,\, A_2:=\log \alpha,\,\text{and}\, A_3:=3\log\rho.
		$$
		By Theorem \ref{Matveev}, we get 
		\begin{align*}
			\log |\Lambda_1|&> -1.4\cdot 30^6\cdot 3^{4.5}\cdot3^2\cdot (1+\log 3)\cdot(1+\log n)\cdot12\log\rho\cdot\log \alpha\cdot 3\log\rho\\
			&>- 3.73\cdot 10^{13}\cdot \log^2\rho\cdot (1+\log n)
		\end{align*}
		which when compared with (\ref{eq11}) gives
		$$
		(\ell-1) \log\rho-\log 3< 3.73\cdot 10^{13}\cdot \log^2\rho\cdot (1+\log n),
		$$
		leading to 
		\begin{equation}\label{eq13}
			\ell \log\rho< 3.74\cdot 10^{13}\cdot \log^2\rho\cdot (1+\log n).
		\end{equation}
\newline
	\end{proof}
	
	\subsubsection{Upper bound for $m \log\rho$ in terms of $n$}
	\begin{lemma}
		All solutions to (\ref{eqfondamentale1}) satisfy
		$$
		m\log\rho< 3.5\cdot 10^{26}\cdot\log^3\rho\cdot (1+\log n)^2.
		$$
	\end{lemma}
	\begin{proof}
		Rewriting (\ref{eq4}), we obtain 
		$$
		(\rho-1)a\alpha^n-(d_1\cdot \rho^\ell-(d_1-d_2))\rho^{m+k}=-(\rho-1)\Pi(n)-(d_2-d_3)\cdot \rho^k-d_3,
		$$
		which shows that 
		\begin{align*}
			\left |(\rho-1)a\alpha^n-(d_1\cdot \rho^\ell-(d_1-d_2))\rho^{m+k} \right|&=\left| -(\rho-1)\Pi(n)-(d_2-d_3)\cdot \rho^k-d_3 \right|\\
			& \leqslant (\rho-1)\cdot \alpha^{-n/2}+(\rho-1)\cdot \rho^k+\rho-1\\
			&< 2\cdot \rho^{k+1}.
		\end{align*}
		Dividing both sides of the inequality by $(d_1\cdot \rho^\ell-(d_1-d_2))\rho^{m+k}$, we have that 
		\begin{equation}\label{eq14}
			\left| \left( \dfrac{(\rho-1)a}{d_1\cdot \rho^\ell-(d_1-d_2)} \right)\cdot\alpha^n\cdot \rho^{-m-k}-1\right|<\dfrac{2\cdot \rho^{k+1}}{(d_1\cdot \rho^\ell-(d_1-d_2))\rho^{m+k}}<\dfrac{2}{\rho^{m-1}}.
		\end{equation}
		Let 
		$$
		\Lambda_2:= \left( \dfrac{(\rho-1)a}{d_1\cdot \rho^\ell-(d_1-d_2)} \right)\cdot\alpha^n\cdot \rho^{-m-k}-1.
		$$
		Using similar argument as in $\Lambda_1$, we apply Theorem \ref{Matveev} on $\Lambda_2$. We notice that $\Lambda_2\neq 0$. If it were, then we would have that 
		$$
		a\alpha^n=\left( \dfrac{d_1\cdot \rho^\ell-(d_1-d_2)}{\rho-1}\right)\cdot \rho^{m+k}.
		$$
		Applying the automorphism $\sigma$ of the Galois group $\mathcal{G}$ on both sides, and taking the absolute values, we obtain 
		
		$$
		1<\left|\left( \dfrac{d_1\cdot \rho^\ell-(d_1-d_2)}{\rho-1}\right)\cdot \rho^{m+k}\right|=|\sigma(a\alpha^n)|=|c\gamma^n|<1,
		$$
		which is false. Therefore, $\Lambda_2\neq 0$. We then proceed to apply Theorem \ref{Matveev}with the following parameters: 
		$$
		\eta_1:=\dfrac{(\rho-1)a}{d_1\cdot \rho^\ell-(d_1-d_2)},\,\eta_2:=\alpha,\, \eta_3:=\rho,\, b_1:=1, \, b_2:=n,\, b_3:=-m-k,\, s=3.
		$$
		Since $m+k<n$, we take $B:=n$. Again taking $\bK:=\bQ(\eta_1,\eta_2,\eta_3)=\bQ(\alpha)$, we have that $d_\bK:=[\bK:\bQ]=3$.
		Next, we use the properties of the logarithmic height to estimate $h(\eta_1)$ as before, and obtain 
		\begin{align*}
			h(\eta_1)&=h\left(   \dfrac{(\rho-1)a}{d_1\cdot \rho^\ell-(d_1-d_2)}\right)\\
			& \leqslant h((\rho-1)a)+ h(d_1\cdot \rho^\ell-(d_1-d_2))\\
			& \leqslant h(\rho-1)+h(a)+\log 2+ h(d_1\cdot \rho^\ell)+h(d_1-d_2)\\
			& \leqslant \dfrac{1}{3}\log 31+ 4\log (\rho-1)+\ell \log\rho+2\log 2\\
			& \leqslant 3.74\cdot 10^{13}\cdot \log^2\rho\cdot (1+\log n)+ 4\log (\rho-1)+ \dfrac{1}{3}\log 31+ 2\log 2\\
			&< 3.75\cdot 10^{13}\cdot \log^2\rho\cdot (1+\log n),
		\end{align*}
		where we used the fact that $\ell \log\rho< 3.74\cdot 10^{13}\cdot \log^2\rho\cdot (1+\log n).$
		We then take 
		$$
		A_1:=1.125\cdot 10^{14}\cdot\log^2\rho\cdot (1+\log n),\, A_2:=\log \alpha,\, \text{and}\, A_3:=3\log\rho.
		$$
		Theorem \ref{Matveev} says that 
		\begin{align*}
			\log |\Lambda_2|&>-1.4\cdot 30^6\cdot 3^{4.5}\cdot 3^2\cdot (1+\log 3)\cdot(1+\log n)\cdot(1.125\cdot 10^{14}\cdot\log^2\rho\cdot \\ & \quad \quad (1+\log n))\cdot\log \alpha\cdot 3\log \rho\\
			&>- 3.49\cdot 10^{26}\cdot\log^3\rho\cdot (1+\log n)^2
		\end{align*}
		which when compared with (\ref{eq14}) gives
		$$
		(m-1)\log\rho-\log 2< 3.49\cdot 10^{26}\cdot\log^3\rho\cdot (1+\log n)^2,
		$$
		which simplifies to 
		\begin{equation}\label{eq20}
			m\log\rho< 3.5\cdot 10^{26}\cdot\log^3\rho\cdot (1+\log n)^2.
		\end{equation}
  \newline

	\end{proof}
	\subsubsection{Upper bound for $\ell+m+k$ and $n$}
	
	Rewriting (\ref{eq4}), we obtain 
	$$
	(\rho-1)a\alpha^n-\left(d_1\cdot \rho^{\ell+m}-(d_1-d_2)\cdot \rho^m-(d_2-d_3)\right)\cdot \rho^k= -(\rho-1)\Pi(n)-d_3,
	$$
	which shows that 
	\begin{align*}
		\left| (\rho-1)a\alpha^n-\left(d_1\cdot \rho^{\ell+m}-(d_1-d_2)\cdot \rho^m-(d_2-d_3)\right)\cdot \rho^k\right|& = |-(\rho-1)\Pi(n)-d_3|\\
		& \leqslant (\rho-1)\cdot \alpha^{-n/2}+(\rho-1)\\
  &<2(\rho-1).
	\end{align*}
	Dividing both sides of the inequality by $(\rho-1)a\alpha^n$, we have that 
	
	\begin{equation}\label{eq15}
		\left|  \left(  \dfrac{d_1\cdot \rho^{\ell+m}-(d_1-d_2)\cdot \rho^m-(d_2-d_3)}{(\rho-1)a}\right)\cdot \alpha^{-n}\cdot \rho^k-1\right|< \dfrac{2}{a\alpha^n}<\dfrac{5}{\alpha^n}.
	\end{equation}
	
	Now, we let 
	$$
	\Lambda_3:=\left(  \dfrac{d_1\cdot \rho^{\ell+m}-(d_1-d_2)\cdot \rho^m-(d_2-d_3)}{(\rho-1)a}\right)\cdot \alpha^{-n}\cdot \rho^k-1.
	$$
	Using similar arguments as in $\Lambda_1$ and $\Lambda_2$, we apply Theorem \ref{Matveev} on $\Lambda_3$. We notice that $\Lambda_3\neq 0.$ If it were, then we would have 
	$$
	a\alpha^n=\left(\dfrac{d_1\cdot \rho^{\ell+m}-(d_1-d_2)\cdot \rho^m-(d_2-d_3)}{\rho-1}\right)\cdot \rho^k.
	$$
	Applying the automorphism $\sigma$ of the Galois group $\mathcal{G}$ on both sides, and taking the absolute values, we obtain
	
	$$
	1<\left| \left(\dfrac{d_1\cdot \rho^{\ell+m}-(d_1-d_2)\cdot \rho^m-(d_2-d_3)}{\rho-1}\right)\cdot \rho^k \right|=|\sigma(a\alpha^n)|=|c\gamma^n|<1,
	$$
	which is false. Therefore, $\Lambda_3\neq 0$. We proceed to apply Theorem \ref{Matveev} with the following parameters :
 \begin{align*}
     &\eta_1:=\left(  \dfrac{d_1\cdot \rho^{\ell+m}-(d_1-d_2)\cdot \rho^m-(d_2-d_3)}{(\rho-1)a}\right),\,\eta_2:=\alpha,\, \eta_3:=\rho,\, \\ &b_1:=1,\, b_2:=-n,\, b_3:=k,\, s:=3.
 \end{align*}
	
	Since $k<n$, we take $B:=n$. Again, taking $\bK=\bQ(\eta_1,\eta_2,\eta_3)=\bQ(\alpha)$, we have that $d_\bK:=[\bK:\bQ]=3$. Nest, we use the properties of the logarithmic height to estimate $h(\eta_1)$ as before, and obtain:
	
	\begin{align*}
		h(\eta_1)&=h\left(  \dfrac{d_1\cdot \rho^{\ell+m}-(d_1-d_2)\cdot \rho^m-(d_2-d_3)}{(\rho-1)a}\right)\\
		& \leqslant h(d_1\cdot \rho^{\ell+m}-(d_1-d_2)\cdot \rho^m-(d_2-d_3))+ h((\rho-1)a)\\
		&  \leqslant h(d_1\cdot \rho^{\ell+m})+ h((d_1-d_2)\cdot \rho^m)+ h(d_2-d_3)+ \log 2+ h(\rho-1)+h(a)\\
		&  \leqslant h(d_1)+ (\ell+m)h(\rho)+ h(d_1-d_2)+mh(\rho) + h(d_2-d_3)+ \log 2+ \\
  & \qquad \quad  h(\rho-1)+h(a)\\
		& \leqslant  h(d_1)+ (\ell+m)h(\rho)+ h(d_1)+h(d_2)+ \log 2+ m h(\rho)+ h(d_2)+ h(d_3)+ \\ & \qquad \quad \log 2+ h(\rho-1)+h(a)\\ 
		& \leqslant 2 h(d_1)+ 2h(d_2)+ h(d_3)+ (2m+\ell) h(\rho)+h(\rho-1)+h(a)+ 2\log 2\\
		&< 6\log (\rho-1) + 2\times 3.5\cdot 10^{26}\cdot\log^3\rho\cdot (1+\log n)^2+ 3.74\cdot 10^{13}\cdot \log^2\rho\cdot \\
  & \qquad \quad  (1+\log n) +  \dfrac{1}{3}\log 31+ 2\log 2\\
		&< 1.4 \cdot 10^{27}\cdot\log^3\rho\cdot (1+\log n)^2.
	\end{align*}
	Note that in above estimate, we have used the estimates from \eqref{eq13} and \eqref{eq20}. We then take 
	$$
	A_1:=4.2\cdot 10^{27}\cdot \log^3\rho\cdot (1+\log n)^2,\, A_2:=\log \alpha,\, \text{and}\, A_3:=3\log \rho.
	$$
	Theorem \ref{Matveev} says that 
	\begin{align*}
		\log |\Lambda_3|& >-1.4\cdot 30^6\cdot 3^{4.5}\cdot 3^2\cdot (1+\log 3)\cdot(1+\log n)\cdot (4.2\cdot 10^{27}\cdot\log^3\rho\cdot (1+\log n)^2)\cdot \\ & \qquad \quad \log \alpha\cdot 3\log\rho\\
		&>- 1.31\cdot 10^{40}\cdot\log^4\rho\cdot (1+\log n)^3
	\end{align*}
	and comparison of this inequality with (\ref{eq15}) gives 
	$$
	n\log \alpha-\log 5<1.31\cdot 10^{40}\cdot\log^4\rho\cdot (1+\log n)^3,
	$$
	which satisfies to 
	$$
	n<2.75\cdot 10^{41}\cdot\log^4\rho\cdot \log^3 n,\quad \text{with}\; 1+\log n<2\log n.
	$$
	Next, we apply Lemma \ref{lem3} that enable us to find an upper bound of $n$, with following parameters 
	$$
	r:=3,\quad L:=n,\quad \text{and} \quad H:=2.75\cdot 10^{41}\cdot\log^4\rho.
	$$
	Therefore, we have 
	$$
	n< 2^3\cdot 2.75\cdot 10^{41}\cdot\log^4\rho\cdot (95.42+4\log\log \rho)^3,
	$$
	which leads to 
	$$
	n<5.6\cdot 10^{48}\cdot \log^7\rho.
	$$
	The above inequality holds because form $\rho\geqslant2,$ we have $95.42+4\log\log\rho<136\log\rho.$ By Lemma \ref{lem2}, we have that 
	$$
	\ell+m+k< n<5.6\cdot 10^{48}\cdot \log^7\rho.
	$$
	This completes the proof of Theorem~\ref{main}.

	\section{The study of the cases $2 \leqslant \rho \leqslant 10$}
	
	In this section, we explicitly determine all the Narayana numbers which are concatenations of three repdigits in base $\rho,$ with $\rho$ between $2$ and $10.$ So our result in this case is Theorem~\ref{Application}. 
	
	\begin{proof}[Proof of Theorem \ref{Application}]
		Then in this range Theorem~\ref{main} allows us to deduce that all solutions to Diophantine equation \eqref{eqfondamentale1} satisfy 
		$$
		\ell+m+k< n<2\times 10^{51}.
		$$
		The next step is therefore the reduction of the upper bound above in order to identify the real range in which the possible solutions of \eqref{eqfondamentale1} are found. For this, we proceed in three steps as follow.\\
		\textbf{Step~1:} Using (\ref{eq11}), let 
		$$
		\Gamma_1:=-\log(\Lambda_1+1)=(\ell+m+k)\log\rho-n\log \alpha-\log \left(  \dfrac{(\rho-1)a}{d_1}\right).
		$$
		Notice that (\ref{eq11}) is rewritten as 
		$$
		\left|  e^{-\Gamma_1}-1 \right|< \dfrac{3}{\rho^{\ell-1}}.
		$$
		Observe that $-\Gamma_1\neq 0$, since $ e^{-\Gamma_1}-1=\Lambda_1\neq 0$. Assume that $\ell \geqslant 4$, then 
		$$
		\left|  e^{-\Gamma_1}-1 \right|< \dfrac{3}{\rho^{\ell-1}}<\dfrac{1}{2}.
		$$
		Therefore, by Lemma \ref{lem4}, we have 
		$$
		\left|  \Gamma_1 \right|<\dfrac{6}{\rho^{\ell-1}}.
		$$
		Substituting $\Gamma_1$ in the above inequality with its value and dividing through by $\log \alpha$, we obtain 
		$$
		\left| (\ell+m+k)\left( \dfrac{\log \rho}{\log \alpha}\right)-n+ \left(  \dfrac{\log (d_1/(\rho-1)a)}{\log \alpha}\right) \right|<\dfrac{6/\log \alpha}{\rho^{\ell-1} }.
		$$ 
		Thus to apply Lemma \ref{Dujella} we can choose 
		$$
		\tau:=\dfrac{\log\rho}{\log \alpha},\quad \mu:=\dfrac{\log (d_1/(\rho-1)a)}{\log \alpha},\quad A:=\dfrac{6}{\log \alpha},\quad B:=\rho,
		$$
		 and $w:=\ell-1$ with $1 \leqslant d_1 \leqslant \rho-1$. Because $\ell+m+k<n<2\times 10^{51},$ we can take $M:=2\times 10^{51}$. So, for the remaining proof, we use \textit{Mathematica} to apply Lemma~\ref{Dujella}. For the computations, if the first convergent $q_t$ such that $q_t > 6M$ does not satisfy the condition $\varepsilon>0,$ then we use the next convergent until we find the one that satisfies the conditions. Thus, we have that 
		\begin{center}
			\begin{tabular}{|c|c|c|c|c|c|c|c|c|c|}
				\hline $\rho$ &2 & 3 & 4 & 5 & 6 & 7 & 8 & 9 &10 \\ \hline \hline 
				$q_t$  & $q_{115}$ & $q_{93}$ & $q_{107}$ & $q_{108}$ & $q_{86}$ & $q_{101}$ & $q_{107}$ & $q_{95}$   & $q_{93}$\\ \hline 
				$\varepsilon \geqslant$ &0.28 & 0.21 & 0.13 & 0.24& 0.03 & 0.03 & 0.02 & 0.05 &0.03\\ \hline
				$\ell-1 \leqslant$  & 183  & 114      & 92       & 78        & 70      & 64       & 61       & 57   &55     \\ \hline
			\end{tabular}
		\end{center}
		\vspace{3mm}
		Therefore,
		$$
		1 \leqslant\ell \leqslant \dfrac{\log((6/\log \alpha)\cdot q_{115}/0.28)}{\log 2}+1 \leqslant 184.
		$$ 
		\textbf{Step~2:} For the next step we have to reduce the upper bound on $m.$ To do this, let us consider
		$$
		\Gamma_2:=-\log (\Lambda_2+1)=(m+k)\log\rho-n\log\alpha+\log\left(\dfrac{d_1\cdot\rho^\ell-(d_1-d_2)}{(\rho-1)a}\right).
		$$
		Thus inequalities \eqref{eq14} become 
		$$
		\left|  e^{-\Gamma_2}-1 \right|< \dfrac{2}{\rho^{m-1}}<\dfrac{1}{2},
		$$
		which holds for $m\geqslant 3.$ It follows from Lemma~\ref{lem4} that 
		\begin{align}\label{Duj}
			\left| (m+k)\left( \dfrac{\log \rho}{\log \alpha}\right)-n+ \left(  \dfrac{\log ((d_1\cdot\rho^\ell-(d_1-d_2))/(\rho-1)a)}{\log \alpha}\right) \right|<\dfrac{4/\log \alpha}{\rho^{m-1} }.
		\end{align}
		Hence, since the conditions of Lemma \ref{Dujella} are satisfied, we may now apply it to inequality \eqref{Duj} with the following data
		$$
		\tau:=\dfrac{\log\rho}{\log \alpha},\quad \mu:=\dfrac{\log ((d_1\cdot\rho^\ell-(d_1-d_2))/(\rho-1)a)}{\log \alpha},\quad A:=\dfrac{4}{\log \alpha},\quad B:=\rho,\quad 
		$$
		and $w:=m-1$ with $1 \leqslant d_1, d_2 \leqslant \rho-1$ and $1 \leqslant \ell \leqslant 184.$ As $m+k< n<2\times 10^{51},$ we can take $M:=2\times 10^{51}.$ With Mathematica we got the following results
		\begin{center}
			\begin{tabular}{|c|c|c|c|c|c|c|c|c|c|}
				\hline $\rho$ &2 & 3 & 4 & 5 & 6 & 7 & 8 & 9 &10 \\ \hline \hline 
				$q_t$  & $q_{115}$ & $q_{93}$ & $q_{107}$ & $q_{108}$ & $q_{86}$ & $q_{101}$ & $q_{107}$ & $q_{95}$   & $q_{93}$\\ \hline 
				$\varepsilon \geqslant$ &0.0009 & 0.0001 & 0.001 & 0.0002& 0.0001 & 0.0006 & 0.0002 & $10^{-6}$ &0.00003\\ \hline
				$m-1  \leqslant$  & 191  & 120       & 95       & 82       & 73      & 66      & 63       & 61 & 58     \\ \hline
			\end{tabular}
		\end{center}
		\vspace{3mm}
		In all cases we can conclude that 
		$$
		1 \leqslant m \leqslant \dfrac{\log((4/\log \alpha)\cdot q_{115}/0.0009)}{\log\rho}+1 \leqslant 192.
		$$ 
		
		\textbf{Step~3:} Finaly, to reduce the bound on $n$ we have to choose
		$$
		\Gamma_3:=\log (\Lambda_3+1)=(m+k)\log\rho-n\log\alpha+\log\left(\dfrac{d_1\cdot\rho^{\ell+m}-(d_1-d_2)\cdot \rho^m-(d_2-d_3)}{(\rho-1)a}\right).
		$$
		By inequalities \eqref{eq15} we have that 
		$$
		\left|  e^{-\Gamma_3}-1 \right|< \dfrac{5}{\rho^n}<\dfrac{1}{2},
		$$
		which is valid for $n\geqslant 3$ and $\rho\geqslant2.$ It follows from Lemma~\ref{lem4} that 
		\begin{align}\label{Duj3}
			\left| k\left( \dfrac{\log \rho}{\log \alpha}\right)-n+   \cfrac{\log \left(\frac{d_1 \cdot \rho^{\ell+m} - (d_1 - d_2) \cdot \rho^m - (d_2 - d_3)}{(\rho - 1) a}\right)}{\log \alpha}, \right|<\dfrac{10/\log \alpha}{\rho^n }.
		\end{align}
		Now we have to apply Lemma~\ref{Dujella} to \eqref{Duj3} by taking the following parameters 
  \begin{align*}
      &\tau:=\dfrac{\log\rho}{\log \alpha},\quad \mu:=\frac{\log \left(\frac{d_1 \cdot \rho^{\ell+m} - (d_1 - d_2) \cdot \rho^m - (d_2 - d_3)}{(\rho - 1) a}\right)}{\log \alpha}, \quad A:=\dfrac{10}{\log \alpha},\quad B:=\rho
  \end{align*}
		
		and $w:=n$  with $1 \leqslant d_1, d_2 \leqslant \rho-1,$  $1 \leqslant \ell \leqslant 184$ and $1 \leqslant m \leqslant 183.$ Using the fact that $k< n<2\times 10^{51},$ we can take $M:=2\times 10^{51}$
		\begin{center}
			\begin{tabular}{|c|c|c|c|c|c|c|c|c|c|}
				\hline $\rho$ &2 & 3 & 4 & 5 & 6 & 7 & 8 & 9 &10 \\ \hline \hline 
				$q_t$  & $q_{115}$ & $q_{93}$ & $q_{107}$ & $q_{108}$ & $q_{86}$ & $q_{101}$ & $q_{107}$ & $q_{95}$   & $q_{93}$\\ \hline 
				$\varepsilon \ge$ &$ 10^{-6}$ & $10^{-7}$ & $10^{-6} $ & $10^{-7}$& $10^{-7}$ & $10^{-8}$ & $10^{-7}$ & $10^{-7}$ & $10^{-8}$\\ \hline
				$n  \leqslant$                    & 201  & 126       & 100       & 87        & 76    & 72       & 67      & 63   &64     \\ \hline
			\end{tabular}
		\end{center}
		\vspace{3mm}
		It follows from the above table that
		$$
		1 \leqslant n \leqslant \dfrac{\log((10/\log \alpha)\cdot q_{115}/10^{-6})}{\log 2} \leqslant 201,
		$$ 
		which is valid for all $\rho$ such that $2 \leqslant \rho \leqslant 10.$ In the light of the above results, we need to check equation \eqref{eqfondamentale1} in the cases $2 \leqslant \rho \leqslant 10$ for $1\le d_1,d_2,d_3\le 9$, $1 \le n \le 201,$ $1 \le k \le 201,$ $1\le \ell\le 184$ and $1\le m \le 192.$ A quick inspection using Maple reveals that Diophantine equation \eqref{eqfondamentale1} has only the solution listed in the statement of Theorem~\ref{Application}. This completes the proof of Theorem~\ref{Application}.
	\end{proof}

	\section{Concluding remarks}
	Specific applications would require a thorough mathematical and physical analysis. Researchers in the field of mathematical physics and interdisciplinary studies may find novel connections by exploring the distribution of repdigits or concatenated sequences in various physical phenomena. The exploration of these concepts contributes to the rich interplay between mathematics and the natural sciences.
	
	Linear recurrent sequences can be applied in various physical contexts, such as modeling population dynamics, heat diffusion, or electrical circuits. The key is to identify a relation where each term in the sequence is a linear combination of previous terms, which often arises naturally in the description of dynamic physical systems.
	
$^*$: Corresponding author \\ \\
$^1$: D\'epartement de Math\'ematiques, Facult\'e des Sciences et Techniques (FaST), Universit\'e de Kara, B.P. 404 Kara, Togo.\\ \\
$^2$: International Chair in Mathematical Physics and Applications ICMPA-UNESCO Chair,  University of Abomey-Calavi, 072 BP. 50 Cotonou, Benin Republic.\\ \\
$^3$: Laboratoire d’Algèbre, de Cryptologie, de G\'eom\'etrie Alg\'ebrique et Applications (LACGAA), Departement de Mathematiques et Informatique, Universit\'e Cheikh Anta Diop de Dakar (UCAD), B.P. 5005 Dakar-Fann, S\'en\'egal.\\ \\
$^4$ : Institute of Mathematics and Physics, Department of Mathematics, University of Abomey-Calavi, 072 BP. 50 Cotonou, Benin Republic.\\ \\
$^5$ : Departement de Physique, Universite de Lome, 01 BP 1515, Lomé, Togo. \\ \\
$^6$ : International Centre for Research and Advanced Studies in Mathematical \& Computer Sciences and Applications (ICRASMCSA), 072 B.P. 50 Cotonou, Benin Republic.\\ \\
Emails: pagdame.tiebekabe@ucaf.edu.sn, adedjnorb1988@gmail.com, npindra@univ-lome.tg, \\ norbert.hounkonnou@cipma.uac.bj/hounkonnou@yahoo.fr.

\end{document}